\documentclass[10pt]{amsart}

\usepackage{mathrsfs}
\usepackage{hyperref}

\theoremstyle{plain}
\newtheorem{teo}{Theorem}[section]

\newtheorem{prop}[teo]{Proposition}
\newtheorem{lema}[teo]{Lemma}

\theoremstyle{definition}

\theoremstyle{remark}
\newtheorem{nota}[teo]{Remark}

\numberwithin{equation}{section}

\newcommand\Ad{\operatorname{Ad}}

\newcommand\Aff{\operatorname{Aff}}

\newcommand\Exp{\operatorname{Exp}}

\newcommand\Iso{\operatorname{Iso}}
\newcommand\Lie{\operatorname{Lie}}

\newcommand\SO{\operatorname{SO}}

\newcommand\SU{\operatorname{SU}}

\newcommand\Tr{\operatorname{Tr}}

\newcommand\iso{\mathfrak{iso}}
\newcommand\so{\mathfrak{so}}
\newcommand\tr{\mathfrak{tr}}
\renewcommand\gg{\mathfrak{g}}

\newcommand\gb{\mathfrak{b}}
\newcommand\gh{\mathfrak{h}}
\newcommand\gk{\mathfrak{k}}
\newcommand\gm{\mathfrak{m}}
\newcommand\gp{\mathfrak{p}}
\newcommand\bbc{\mathbb{C}}
\newcommand\bbv{\mathbb{V}}
\newcommand\bbw{\mathbb{W}}
\newcommand\bbr{\mathbb{R}}
\newcommand\cc{\mathscr{C}}
\newcommand\cd{\mathscr{D}}

\newcommand\cf{\mathscr{F}}

\begin{document}

\title[The uniqueness of the canonical connection]{A note
  on the uniqueness of the canonical connection of a naturally
  reductive space}

\author{Carlos Olmos}

\author{Silvio Reggiani} 

\address{Facultad de Matem\'atica, Astronom\'ia y F\'\i sica,
  Universidad Nacional de C\'ordoba, Ciudad Universitaria, 5000
  C\'ordoba, Argentina}

\email{olmos@famaf.unc.edu.ar \qquad reggiani@famaf.unc.edu.ar}

\date{\today}



\thanks{Supported by Universidad Nacional de C\'ordoba and
  CONICET. Partially supported by ANCyT, Secyt-UNC and CIEM}

\begin{abstract}
We extend the result in J.\ Reine Angew.\ Math.\ \textbf{664}, 29--53,
to the non-compact case. Namely, we prove that the canonical
connection on a simply connected and irreducible naturally reductive
space is unique, provided the space is not a sphere, a compact Lie
group with a bi-invariant metric or its symmetric dual. In particular,
the canonical connection is unique for the hyperbolic space when the
dimension is different from three. We also prove that the canonical
connection on the sphere is unique for the symmetric
presentation. Finally, we compute the full isometry group (connected
component) of a compact and locally irreducible naturally reductive
space.
\end{abstract}

\maketitle

\section{Introduction and preliminaries}

\'Elie Cartan, in the 1920s, asked for linear connections, on a given
Riemannian space, that adapt to the geometry in a more suitable way
than the Levi-Civita connection \cite{cartan-1924}. He proposed to
study the so-called connections with skew-torsion. Such connections
are characterized by the property of having parallel metric tensor and
the same geodesics as the Levi-Civita connection.

Spaces with skew-torsion have an increasing interest in recent years
because of their  applications to theoretical physics (see
\cite{agricola-2006}). A distinguished family of Riemannian spaces with
skew-torsion are the naturally reductive spaces. In fact, the
canonical connection $\nabla^c$ of a naturally reductive space $M=G/H$
provides a metric connection and has skew-torsion $T = -2(\nabla -
\nabla^c)$, where $\nabla$ is the Levi-Civita connection. If $M$ is a
symmetric space, then the Levi-Civita connection is a canonical
connection. 

In a naturally reductive space one has that $\nabla^c R = 0$ and
$\nabla^c T = 0$, where $R$ is the Riemannian curvature tensor. More
generally, any $G$-invariant tensor on $M$ must be parallel with
respect to the canonical connection.

In \cite{olmos-reggiani-2012} it was proved that the canonical
connection of a (locally irreducible) compact naturally reductive
space is unique, provided  the space is  different from the following
symmetric spaces: spheres, real projective spaces and compact Lie
groups with a bi-invariant metric.

The proof given in \cite{olmos-reggiani-2012} uses strongly the
compactness assumption (besides the so-called skew-torsion holonomy
theorem). Namely, it makes use of a decomposition theorem for compact
homogeneous spaces which is false in the non-compact case (however,
such a decomposition theorem was crucial in the proof the skew-torsion
holonomy theorem).

The purpose of this note is to prove that the canonical connection is
unique also for simply connected (irreducible) non-compact naturally
reductive spaces, with the  only exceptions of   dual symmetric spaces
of compact  Lie groups. In particular, the canonical connection is
unique for any real hyperbolic space $M = H^n$ with $n \neq 3$ (in
contrast with the compact case where many spheres are excluded).

Observe that the main result of this article, stated precisely in
Theorem \ref{unica}, also has a local version, since a canonical
connection on a naturally reductive space lifts to the universal
cover.

In order to prove Theorem~\ref{unica}, we use some auxiliary facts
that we want to mention. Namely,  that the real hyperbolic space $H^n$
admits a unique naturally reductive presentation (the symmetric pair
presentation). This allows to prove that the canonical connection on
$H^n$ is unique for all $n \neq 3$. To do this, we use that the
canonical connection on the sphere $S^n$, with $n \neq 3$, is unique
if we fix the symmetric presentation $S^n = \SO(n + 1)/\SO(n)$, since
from Hodge theory there are no non-trivial parallel $3$-forms (see
Remark~\ref{hodge}).

Finally, in Section~\ref{sec:3}, we explicitly compute the isometry
group of a compact and locally irreducible naturally reductive
space. This extends the result in \cite{reggiani-2010} for normal
homogeneous spaces (and known results by Onishchik
\cite{onishchik-1992} and Shankar \cite{shankar-2001} on isometry
groups of homogeneous spaces).

\subsection{Skew-torsion holonomy systems}

In a previous work \cite{olmos-reggiani-2012} we deal with the concept
of skew-torsion holonomy systems, which are a variation of the
so-called holonomy systems introduced by J.\ Simons in
\cite{simons-1962}. Skew-torsion holonomy systems arise in a natural
way and in a geometric context, by considering the difference tensor
between two metric connections which have the same geodesics as the
Levi-Civita connection. 

We say that a triple $[\bbv, \Theta, G]$ is a \emph{skew-torsion
  holonomy system} provided $\bbv$ is an Euclidean space, $G$ is a
connected Lie subgroup of $\SO(\bbv)$, and $\Theta$ is a totally
skew-symmetric $1$-form on $\bbv$ which takes values in the Lie
algebra $\gg$ of $G$ (i.e., $(x, y, z) \mapsto \langle\Theta_xy,
z\rangle$ is an algebraic $3$-form on $\bbv$). We say that $[\bbv,
  \Theta, G]$ is \emph{irreducible} if $G$ acts irreducibly on $\bbv$,
\emph{transitive} if $G$ is transitive on the sphere of $\bbv$, and
\emph{symmetric} if $g_*(\Theta) = \Theta$ for all $g \in G$, where
$g_*(\Theta)_x = g \circ \Theta_{g^{-1}(x)} \circ g^{-1}$. 

The main result on skew-torsion holonomy systems is analogous to
Simons holonomy theorem for classical holonomy systems. Such a result
is actually stronger because transitive cases cannot occur others than
the full orthogonal group. 

\begin{teo}[Skew-torsion Holonomy Theorem
    \cite{nagy-2007,olmos-reggiani-2012}] 
Let $[\bbv, \Theta, G]$, $\Theta \neq 0$, be an irreducible
skew-torsion holonomy system with $G \neq \SO(\bbv)$. Then $[\bbv,
  \Theta, G]$ is symmetric and non-transitive. Moreover, 
\begin{enumerate} 
\item $(\bbv, [\cdot,\cdot])$ is an orthogonal simple Lie algebra, of
  rank at least $2$, with respect to the bracket $[x,y] = \Theta_xy$; 
\item $G = \Ad(H)$, where $H$ is the connected Lie group associated to
  the Lie algebra $(\bbv, [\cdot,\cdot])$; 
\item $\Theta$ is unique, up to a scalar multiple. 
\end{enumerate}
\end{teo}

\section{The uniqueness of the canonical connection}

In this section we prove a uniqueness result for canonical connections
on naturally reductive spaces, compact or not. 

\begin{teo}\label{unica} 
Let $M$ be a simply connected and irreducible naturally reductive
space. Assume that $M$ is not (globally) isometric to a sphere, nor to
a Lie group with a bi-invariant metric or its symmetric dual. Then,
the canonical connection on $M$ is unique. 
\end{teo}

Observe that, in particular, the above theorem says that the canonical
connection is unique for the real hyperbolic space $H^n$ for all $n
\neq 3$. When $n = 3$, $H^3$ is the symmetric dual of $S^3 = \SU(2)$,
and in this case $H^3$ admits a line of canonical connections (see
Remark~\ref{hodge}). 

Before giving the proof of Theorem \ref{unica} we fix some notation
and we state some basic results we will need. 

Let $M = G/G_p$, $p \in M$, be a naturally reductive space. That is,
assume that $M$ carries a $G$-invariant metric and the Lie algebra of
$G$ admits a decomposition $\gg = \gg_p \oplus \gm$, where $\gg_p =
\Lie(G_p)$ and $\gm$ is an $\Ad(G_p)$-invariant subspace such that the
geodesics through $p$ are given by 
$$\Exp(tX) \cdot p, \qquad X \in \gm.$$ 
That is to say, Riemannian geodesics coincide with
$\nabla^c$-geodesics, where $\nabla^c$ is the canonical connection
associated with the above mentioned reductive decomposition. 

Suppose that $\nabla^{c'}$ is another canonical connection on $M$
(associated with another naturally reductive presentation or another
reductive decomposition). It follows from \cite[Section
  6]{olmos-reggiani-2012} that 
$$\Theta = (\nabla^{c'} - \nabla^c)_p,$$ 
that is the difference between $\nabla^{c'}$ and $\nabla^c$, evaluated
at $p$, is a totally skew-symmetric $1$-form on $T_pM$ which takes
values in the full isotropy subalgebra $\gh = \Lie(\Iso(M)_p)$. Hence,
$[T_pM, \Theta, H]$ is a skew-torsion holonomy system, where $H =
(\Iso(M)_p)^o$ is the connected component of the full isotropy
subgroup at $p$.  

Let $\tilde\gh$ be the linear span of $\{h_*(\Theta)_v: h \in H, v \in
T_pM\}$. We have that $\tilde\gh$ is an ideal of $\gh$. Let $\tilde H$
be the connected Lie subgroup of $H$ with Lie algebra $\tilde
\gh$. From \cite[Section 2]{olmos-reggiani-2012} there exist
decompositions  
\begin{equation}\label{eq:space}
T_pM = \bbv_0 \oplus \bbv_1 \oplus \cdots \oplus \bbv_k \qquad
\text{(orthogonal sum)} 
\end{equation}
and
\begin{equation}\label{eq:group}
\tilde H = H_1 \times \cdots \times H_k \qquad \text{(almost direct
  product)} 
\end{equation}
such that $H_i$ acts trivially on $\bbv_j$ if $i \neq j$ (in
particular, $\bbv_0$ is the set of fixed vectors of $\tilde H$) and
$H_i$ acts irreducibly on $\bbv_i$ with $\cc_i(\gh_i) = \{0\}$, where 
$$\cc_i(\gh_i) := \{B \in \so(\bbv_i): [B, \gh_i] = 0\}.$$ 
Moreover, we have that $H$ splits as
$$H = H_0 \times \tilde H = H_0 \times H_1 \times \cdots \times H_k,$$ 
where $H_0$ acts only on $\bbv_0$ (and it could be arbitrary). In
fact, any skew-torsion holonomy system can be decomposed in this way
(see \cite{olmos-reggiani-2012} and also
\cite{agricola-friedrich-2004, nagy-2007}).

In order to prove Theorem \ref{unica}, we will make use of the
following basic facts. 

\begin{lema}[see \cite{olmos-reggiani-2012}]\label{5.1-OR} 
Let $M = G/G_p$ be a Riemannian homogeneous manifold, let $H$ be a
normal subgroup of $G_p$ and let $\bbw$ be the subspace of $T_pM$
defined by 
$$\bbw = \{v \in T_pM: dh(v) = v \text{ for all $h \in H$}\}.$$ 
Then $\bbw$ is $G_p$-invariant. Moreover, if $\cd$ is the
$G$-invariant distribution on $M$ defined by $\cd(p) = \bbw$, then
$\cd$ is integrable with totally geodesic leaves (or, equivalently,
$\cd$ is autoparallel). 
\end{lema}

\begin{lema}\label{killing}
Let $M = G/G_p$ be a naturally reductive space. If $X$ is a
$G$-invariant field on $M$, then $X$ is a Killing field. 
\end{lema}

\begin{proof}
Let $X$ be a $G$-invariant field on $M$ and let $D = \nabla -
\nabla^c$ be the difference tensor between the Levi-Civita connection
and a canonical connection on $M$ associated with a reductive
decomposition $\gg = \gg_p \oplus \gm$. Since $X$ is $G$-invariant,
then $X$ is $\nabla^c$-parallel (since $\nabla^c$ is $G$-invariant and
the $\nabla^c$-parallel transport along the geodesic $\Exp(tZ) \cdot
p$, $Z \in \gm$, is given by $\Exp(tZ)_*$). So, $\nabla X = DX$ is
skew-symmetric and this implies that $X$ is a Killing field. 
\end{proof}

\begin{proof}[Proof of Theorem \ref{unica}]
We keep the notation from the previous paragraphs. We have
decompositions $T_pM = \bbv_0 \oplus \bbv_1 \oplus \cdots \oplus
\bbv_k$ (orthogonal) as in \ref{eq:space} and $H = H_0 \times \tilde H
= H_0 \times H_1 \times \cdots \times H_k$ as in \ref{eq:group}. Let
$\bbw_0$ be the set of fixed vectors of $H$ in $T_pM$, via the
isotropy representation. So, $\bbw_0 \oplus \bbv_1$ is the set of
fixed vectors of $H^1 = H_0 \times H_2 \times \cdots \times H_k$ and
hence, by Lemma \ref{5.1-OR}, it induces the $G$-invariant
autoparallel distribution $\cd^1$ defined by $\cd^1(p) = \bbw_0 
\oplus \bbv_1$.

Let $\cd_0$ the $G$-invariant autoparallel distribution defined by
$\cd_0(p) = \bbw_0$. The key factor in the proof is to show that
$\cd_0$ is parallel along $\cd^1$, and then make use of the
skew-torsion holonomy theorem. Since $\cd_0$ is $G$-invariant we only
have to prove that $\cd_0$ is parallel at $p$ (along $\cd^1$). 

Let $S^1(p)$ be the maximal connected integral manifold of $\cd^1$
which contains $p$. That is, $S^1(p)$ is the set of fixed points of
$H^1$ on $M$ (connected component). It is not difficult to see that
$S^1(p)$ is an extrinsic homogeneous submanifold under the action of
the group 
$$G^1(p) = \{g \in G: g(S^1(p)) = S^1(p)\} = \{g \in G: g(p) \in
S^1(p)\}$$ 
with effective isotropy $H_1$. Recall that the metric on $S^1(p)$ is
naturally reductive, since $S^1(p)$ is a totally geodesic submanifold
of $M$. 

Let $X \in \bbw_0$ and let $\tilde X$ be the $G^1(p)$-invariant field
on $S^1(p)$ such that $\tilde X(p) = X$, or equivalently, the
restriction to $S^1(p)$ of the $G$-invariant field on $M$ with initial
condition $X$. It follows from Lemma \ref{killing} that $\tilde X$ is
a Killing field and hence, its derivative $\nabla \tilde X$ is
skew-symmetric. 

Observe that if $h \in H_1$ and $v \in \bbw_0 \oplus \bbv_1 \simeq
T_pS^1(p)$, then 
$$dh(\nabla_v\tilde X) = \nabla_{dh(v)}h_*(\tilde X) =
\nabla_{dh(v)}\tilde X.$$ 
Then $(\nabla \tilde X)_p$ commutes with $H_1$ (via the isotropy
representation) and so, it leaves $\bbw_0$ and $\bbv_1$
invariant. Since $\cc_1(\gh_1) = \{0\}$ we have that $(\nabla\tilde
X)_p|_{\bbv_1} \equiv 0$, and therefore $\nabla_v\tilde X \in \bbw_0$
for all $v \in \bbw_0 \oplus \bbv_1$. This implies that $\cd_0$ is
parallel along $\cd^1$. Then, $\cd_1$ is parallel along $\cd^1$ and
hence, $\cd_1$ is autoparallel on $M$, since $\cd^1$ is
autoparallel. On the other side, we have that $\cd_1^{\bot}$ is also
an autoparallel distribution on $M$, since $\cd_1^\bot(p) = \bbv_0
\oplus \bbv_2 \oplus \cdots \oplus \bbv_k$ is the set of fixed vectors
of $H_1$, and therefore $M$ splits off, unless these distributions are
trivial. 

Finally, we reach two possibilities:
\begin{enumerate}
\item $T_pM = \bbv_0$ and $H = (\Iso(M)_p)^o = H_0$, or 
\item $T_pM = \bbv_1$ and $H = (\Iso(M)_p)^o = H_1$.
\end{enumerate}

In the first case, we have that the group $\tilde H$ spanned by
$\nabla^{c'} - \nabla^c$ is trivial, and then we conclude that
$\nabla^c = \nabla^{c'}$. 

In the second case, we have two possibilities again. Firstly, if $H_1$
is transitive on the unit sphere of $T_pM$, then, by using the
skew-torsion holonomy theorem, we have that $H_1 = (\Iso(M)_p)^o =
\SO(T_pM)$. So, it is standard to see that $M = S^n$ or $M = H^n$. See
Proposition \ref{hiperbolico} and Remark \ref{hodge} below to exclude
the hyperbolic case when $n \neq 3$. On the other hand, if $H_1$ is
not transitive on the sphere, the skew-torsion holonomy theorem says
that $H_1$ acts on $T_pM$ as the adjoint representation of a simple
and compact Lie group. If $M$ is compact, it follows from the
classification of strongly isotropy irreducible spaces, given by
J.\ Wolf in \cite{wolf-1968} (see \cite[Appendix]{olmos-reggiani-2012}
for a conceptual proof), that $M$ is a Lie group with a bi-invariant
metric. If $M$ is non-compact, then $M$ turns out a symmetric space
(since non-compact isotropy irreducible spaces must be symmetric
\cite{besse-1987, wang-ziller-1991}). If $\nabla^c \neq \nabla^{c'}$,
we have, by taking the symmetric dual, that $M^*$ is isometric to a
Lie group with a bi-invariant metric. In fact, it is not difficult to
see that there is a one-one correspondence between canonical
connections on $M$ and canonical connections on $M^*$ (see Remark
\ref{dual}). 

This completes the proof of Theorem \ref{unica}.
\end{proof}

\begin{prop}\label{hiperbolico}
The real hyperbolic space $H^n$ admits a unique naturally reductive
presentation, the symmetric pair decomposition $H^n = \SO(n + 1,
1)^o/\SO(n)$. 
\end{prop}

\begin{proof}
Let $G$ be a connected Lie subgroup of $\Iso(H^n)$ which acts
transitively on $H^n$ and such that $H^n = G/H$ is a naturally
reductive space. If $G$ is semisimple, it is standard to show that $G
= \Iso(H^n)^o = \SO(n + 1, 1)^o$. In fact, let $K$ be a maximal
compact subgroup of $G$. So, $K$ has a fixed point, say $p$. We may
assume that $H$ is the isotropy group at $p$. So $H = K$, since $K$ is
maximal. Hence, $(G,H)$ is presentation of $H^n$ as an effective
Riemannian symmetric pair, and therefore $G = \SO(n + 1, 1)^o$
(otherwise, $H^n$ would have two different presentations as an
effective Riemannian symmetric pair). 

If $G$ is not semisimple, then $G$ has a nontrivial normal abelian Lie
subgroup $A$. It is a well-known fact that, either $A$ fixes a unique
point at infinity or $A$ translates a unique geodesic. If $A$
translates a unique geodesic $\gamma(t)$, then $G$ leaves $\gamma$
invariant, since $A$ is a normal subgroup of $G$, and so $G$ cannot be
transitive, which is a contradiction. So, let $q_\infty$ be the unique
point at infinity which is fixed by $A$, and let $\cf$ be the
foliation on $H^n$ by parallel horospheres centered at $q_\infty$. So,
we have that $A$ leaves $\cf$ invariant, and hence $G$ does. Let $p
\in H^n$ and let $\cf_p$ be the horosphere through $p$. Denote by
$\tilde G$ the connected component of the subgroup of $G$ which leaves
$\cf_p$ invariant. Then $\tilde G$ is transitive on $\cf_p$. Hence,
since $H^n$ is naturally reductive with respect to the decomposition
$G/H$, each horosphere must be totally geodesic, a contradiction. 
\end{proof}

\begin{nota}\label{hodge}
Let us consider the sphere $S^n = \SO(n+1)/\SO(n)$. Then, for all $n
\neq 3$, the Levi-Civita connection is the unique canonical connection
on $S^n$ associated with this naturally reductive decomposition. In
fact, if $\nabla^c$ is another canonical connection on $S^n$ then, the
difference tensor $D = \nabla - \nabla^c$ induces a
$\SO(n+1)$-invariant $3$-form $\omega(x,y,z) = \langle D_xy, z
\rangle$. Since $\omega$ is invariant, $\omega$ is a harmonic $3$-form
on $S^n$. From Hodge theory, $\omega$ represents a nontrivial
cohomology class of order $3$ of the sphere $S^n$. This yields a
contradiction, unless $n = 3$. 

As a consequence, it follows from Proposition \ref{hiperbolico} and
the next remark that the real hyperbolic space $H^n$ admits a unique
canonical connection for all $n \neq 3$. If $n=3$, $H^3$ is the dual
symmetric space of the compact Lie group $S^3 \simeq \SU(2)$, and
therefore it admits exactly a line of canonical connections (see
\cite[Remark~6.1]{olmos-reggiani-2012}).
\end{nota}

\begin{nota}\label{dual}
Let $M = G/K$ be a symmetric space with associated Cartan
decomposition $\gg = \gk \oplus \gp$. Then, there is a one-one
correspondence between canonical connections on $M$ and canonical
connections on the dual $M^* = G^*/K$. In fact, assume that $M$ admits
a canonical connection $\nabla^c$ associated with a reductive
decomposition $\gg = \gk \oplus \gm$. Let $\gg^* = \gk \oplus i\,\gp$
be the Lie algebra of $G^*$, regarded as a subspace of the
complexification $\gg^\bbc$ of $\gg$. It is clear that $\gm^*$ (the
subspace of $\gg^*$ induced by $\gm$, via the vector spaces
isomorphism $\gg^* \simeq \gg$) is and $\Ad^*(K)$-invariant subspace
such that the geodesics through $p = eK$ are given by 1-parameter
subgroups with initial velocities in $\gm^*$. So, $\nabla^c$
corresponds to a unique canonical connection on $M^*$. 
\end{nota}

\section{The isometry group of compact naturally reductive spaces} 
\label{sec:3}

Let $M = G/H$ be a compact and locally irreducible naturally reductive
space and let $\nabla^c$ be the canonical connection associated with
the reductive decomposition $\gg = \gh \oplus \gm$. Assume that $M
\neq S^n$, $M \neq \bbr P^n$. Then, from
\cite[Theorem~1.1]{olmos-reggiani-2012} we have that $\Iso(M)^o =
\Aff(M, \nabla^c)^o$, where $\Aff(M, \nabla^c)^o$ is the connected
component of the affine group of $\nabla^c$ (i.e., the subgroup of
diffeomorphisms of $M$ which preserve~$\nabla^c$).

By making use of Lemma \ref{killing} and some arguments in
\cite{reggiani-2010} one can obtain the connected component of the
isometry group of $M$. Actually, it is possible to simplify such
arguments. 

In fact, let $\Tr(M, \nabla^c)$ be the group of transvections of
$\nabla^c$, that is, the connected Lie subgroup of $\Aff(M,
\nabla^c)^o$ with Lie algebra $\tr(M, \nabla^c) = [\gm, \gm] + \gm$
(not a direct sum, in general). Recall that $\Tr(M, \nabla^c)$ is a
normal subgroup of $\Aff(M, \nabla^c)^o$. As it is done in
\cite{reggiani-2010} for normal homogeneous spaces, we have that $G =
\Tr(M, \nabla^c)$ and thus, $G$ is a normal subgroup of $\Aff(M,
\nabla^c)^o$. (In fact, $\tr(M, \nabla^c)$ is an ideal of $\gg$, then
if $\tr(M, \nabla^c) \neq \gg$, since $M$ is compact, one can take a
complementary ideal of the transvection algebra in $\gg$, which must
be contained in the isotropy algebra. This is a contradiction, since
we assume that $G$ acts effectively on $M$.) 

Now, since $G$ is a normal subgroup of $\Iso(M)^o = \Aff(M,
\nabla^c)^o$, we can write 
$$\iso(M) = \gg \oplus \gb,$$
where $\gb$ is a complementary ideal of $\gg$ in $\iso(M)$ (recall
that $M$ is compact, and hence $\Iso(M)^o$ is also compact). Note that
elements of $\gb$ correspond to $G$-invariant fields on $M$, which are
Killing fields by Lemma \ref{killing} (but not any $G$-invariant field
belongs to $\gb$, in principle). 

We can summarize this fact as follows.

\begin{teo}\label{isometrias}
Let $M = G/H$ be a compact naturally reductive space. Assume that $M$
is locally irreducible and that $M$ is not (globally) isometric to the
sphere $S^n$ nor to the real projective space $\bbr P^n$. Then the
connected component of the isometry group of $M$ is given by 
$$\Iso(M)^o = G_{\mathrm{ss}} \times K \qquad \text{(almost direct product),}$$
where $G_{\mathrm{ss}}$ is the semisimple part of $G$ and $K$ is the
connected subgroup of $\Iso(M)$ whose Lie algebra consists of the
$G$-invariant fields. In particular, $\Iso(M)$ is semisimple if and
only if $K$ is semisimple. 
\end{teo}

\begin{nota}
In the notation of Theorem \ref{isometrias}, $K$ can be identified
with (the connected component of) the set of fixed points of the
isotropy group $H$, acting simply and transitively by right
multiplication. Moreover, just by coping the argument in \cite[Theorem
  1.4]{reggiani-2010} we get that the set of fixed points of the full
isotropy group $(\Iso(M)^o)_p$ is a torus. 
\end{nota}


\end{document}